\documentclass[10pt]{article}

\usepackage[utf8]{inputenc}
\usepackage[T1]{fontenc}

\usepackage{epsf}
\usepackage{amsmath}

\allowdisplaybreaks

\usepackage[showframe=false]{geometry}
\usepackage{changepage}

\usepackage{epsfig}
\usepackage{amssymb}

\usepackage{amsthm}
\usepackage{setspace}
\usepackage{cite}
\usepackage{mcite}

\usepackage{algorithmic}  
\usepackage{algorithm}

\usepackage{shadow}
\usepackage{fancybox}
\usepackage{fancyhdr}

\usepackage{color}
\usepackage[usenames,dvipsnames,svgnames,table]{xcolor}
\newcommand{\bl}[1]{\textcolor{blue}{#1}}

\definecolor{mypurple}{rgb}{.4,.0,.5}

\usepackage[hyphens]{url}

\usepackage[colorlinks=true,
            linkcolor=black,
            urlcolor=blue,
            citecolor=purple]{hyperref}

\usepackage{breakurl}

\def\x{{\bf x}}

\def\x{{\mathbf x}}

\def\x{{\bf x}}

\def\a{{\bf a}}
\def\b{{\bf b}}
\def\c{{\bf c}}

\def\h{{\bf h}}

\def\be{\begin{equation}}
\def\ee{\end{equation}}
\def\ba{\left[\begin{array}}
\def\ea{\end{array}\right]}

\def\x{{\bf x}}

\def\a{{\bf a}}
\def\b{{\bf b}}
\def\c{{\bf c}}

\def\1{{\bf 1}}

\def\g{{\bf g}}
\def\0{{\bf 0}}

\def\erfc{\mbox{erfc}}

\def\mR{{\mathbb R}}

\def\mE{{\mathbb E}}

\def\mP{{\mathbb P}}

\def\lp{\left (}
\def\rp{\right )}

\def\x{{\bf x}}

\def\x{{\mathbf x}}

\def\x{{\bf x}}

\def\a{{\bf a}}
\def\b{{\bf b}}
\def\c{{\bf c}}

\def\h{{\bf h}}

\def\be{\begin{equation}}
\def\ee{\end{equation}}
\def\ba{\left[\begin{array}}
\def\ea{\end{array}\right]}

\def\x{{\bf x}}

\def\a{{\bf a}}
\def\b{{\bf b}}
\def\c{{\bf c}}

\def\({\left (}
\def\){\right )}

\def\1{{\bf 1}}

\def\g{{\bf g}}
\def\0{{\bf 0}}

\usepackage{xcolor}
\usepackage{color}

\definecolor{darkgreen}{rgb}{0, 0.4,0}

\definecolor{purplebrown}{rgb}{0.5,0.1,0.6}

\definecolor{ultclupcol}{rgb}{0.1,0.5,0.5}

\definecolor{mytrycolor}{rgb}{0.5,0.7,0.2}

\definecolor{ultclupcola}{rgb}{.5,0,.5}

\definecolor{shadebrown}{rgb}{0.1,0.1,0.9}
\definecolor{lightblue}{rgb}{0.2,0,1}

\usepackage{fancybox}
\usepackage{graphicx}
\usepackage{epstopdf}
\usepackage{epsfig}
\usepackage{wrapfig}
\usepackage{subfigure}

\usepackage{xcolor}
\usepackage{tcolorbox}
\tcbuselibrary{skins}

\newtcbox{\xmybox}{on line,
arc=7pt,
before upper={\rule[-3pt]{0pt}{10pt}},boxrule=0pt,
boxsep=0pt,left=6pt,right=6pt,top=0pt,bottom=0pt,enhanced, coltext=blue, colback=white!10!yellow}

\newtcbox{\xmyboxa}{on line,
arc=7pt,
before upper={\rule[-3pt]{0pt}{10pt}},boxrule=0pt,
boxsep=0pt,left=6pt,right=6pt,top=0pt,bottom=0pt,enhanced, colback=white!10!yellow}

\newtcbox{\xmyboxb}{on line,
arc=7pt,
before upper={\rule[-3pt]{0pt}{10pt}},boxrule=1pt,colframe=darkgreen!100!blue,
boxsep=0pt,left=6pt,right=6pt,top=0pt,bottom=0pt,enhanced, colback=white!10!yellow}

\newtcbox{\xmyboxc}{on line,
arc=7pt,
before upper={\rule[-3pt]{0pt}{10pt}},boxrule=.7pt,colframe=blue!100!blue,
boxsep=0pt,left=6pt,right=6pt,top=0pt,bottom=0pt,enhanced, coltext=blue, colback=white!10!yellow}

\newtcbox{\xmytboxa}{on line,
arc=7pt,
before upper={\rule[-3pt]{0pt}{10pt}},boxrule=.0pt,colframe=pink!50!yellow,
boxsep=0pt,left=6pt,right=6pt,top=0pt,bottom=0pt,enhanced, coltext=white, colback=blue!40!red}

\newtcbox{\xmytboxb}{on line,
arc=7pt,
before upper={\rule[-3pt]{0pt}{10pt}},boxrule=.0pt,colframe=pink!50!yellow,
boxsep=0pt,left=6pt,right=6pt,top=0pt,bottom=0pt,enhanced, coltext=white, colback=white!40!green}

\makeatletter
\newcommand\subsubsubsection{\@startsection{paragraph}{4}{\z@}{-2.5ex\@plus -1ex \@minus -.25ex}{1.25ex \@plus .25ex}{\normalfont\normalsize\bfseries}}
\newcommand\subsubsubsubsection{\@startsection{subparagraph}{5}{\z@}{-2.5ex\@plus -1ex \@minus -.25ex}{1.25ex \@plus .25ex}{\normalfont\normalsize\bfseries}}
\makeatother

\newtheorem{theorem}{Theorem}

\newtheorem{lemma}{Lemma}

\begin{document}

\begin{singlespace}

\title {Exact objectives of random linear programs and mean widths of random polyhedrons 
}
\author{
\textsc{Mihailo Stojnic
\footnote{e-mail: {\tt flatoyer@gmail.com}} }}
\date{}
\maketitle

\centerline{{\bf Abstract}} \vspace*{0.1in}

We consider \emph{random linear programs} (rlps) as a subclass  of \emph{random optimization problems} (rops) and study their typical behavior. Our particular focus is on appropriate linear objectives which connect the rlps to the mean widths of random polyhedrons/polytopes. Utilizing the powerful machinery of \emph{random duality theory} (RDT) \cite{StojnicRegRndDlt10}, we obtain, in a large dimensional context, the exact characterizations of the program's objectives. In particular, for any $\alpha=\lim_{n\rightarrow\infty}\frac{m}{n}\in(0,\infty)$, any unit vector $\mathbf{c}\in{\mathbb R}^n$, any fixed $\mathbf{a}\in{\mathbb R}^n$, and $A\in {\mathbb R}^{m\times n}$ with iid standard normal entries, we have
\begin{eqnarray*}
 \lim_{n\rightarrow\infty}{\mathbb P}_{A} \lp  (1-\epsilon) \xi_{opt}(\alpha;\mathbf{a})
 \leq  \min_{A\mathbf{x}\leq \mathbf{a}}\mathbf{c}^T\mathbf{x} \leq (1+\epsilon) \xi_{opt}(\alpha;\mathbf{a}) \rp
\longrightarrow 1,
\end{eqnarray*}
where
\begin{equation*}
\xi_{opt}(\alpha;\mathbf{a})  \triangleq  \min_{x>0} \sqrt{x^2-   x^2 \lim_{n\rightarrow\infty}  \frac{\sum_{i=1}^{m}   \left ( \frac{1}{2} \left (\left ( \frac{\mathbf{a}_i}{x}\right )^2 + 1\right ) \mbox{erfc}\left( \frac{\mathbf{a}_i}{x\sqrt{2}}\right ) - \frac{\mathbf{a}_i}{x} \frac{e^{-\frac{\mathbf{a}_i^2}{2x^2}}}{\sqrt{2\pi}} \right )
   }{n} }.
\end{equation*}
For example, for $\a=\mathbf{1}$, one uncovers
\begin{equation*}
 \xi_{opt}(\alpha)
 =
  \min_{x>0} \sqrt{x^2-   x^2 \alpha \left ( \frac{1}{2} \left ( \frac{1}{x^2} + 1\right ) \erfc \left ( \frac{1}{x\sqrt{2}}\right ) - \frac{1}{x} \frac{e^{-\frac{1}{2x^2}}}{\sqrt{2\pi}} \right ) }.
\end{equation*}
Moreover, $2 \xi_{opt}(\alpha;\mathbf{1})$ is precisely the concentrating point of the mean width of the polyhedron $\{\x|A\x\leq \mathbf{1}\}$.

\vspace*{0.25in} \noindent {\bf Index Terms: Random linear programs; Polyhedron/polytope width; Random duality theory}.

\end{singlespace}

\section{Introduction}
\label{sec:back}

Given a function $f(\x):\mR^n\rightarrow\mR$ a generic linearly constrained optimization problems have the following form:
\begin{eqnarray}
\min_{\x\in\mR^{n}} & & f(\x)\nonumber \\
\mbox{subject to} & & A\x\leq \a \nonumber \\
& & B\x=\b, \label{eq:lincons}
\end{eqnarray}
where $A\in\mR^{m_1\times n}$ is an $m_1\times n$ matrix, $B\in\mR^{m_2\times n}$ is an $m_2\times n$ matrix, $\a\in\mR^{m_1\times 1}$ an $m_1\times 1$ vector, and $\b\in\mR^{m_2\times 1}$ is an $m_2\times 1$ vector. If, additionally, one also has that the objective is linear, i.e. for some $\c\in\mR^{n}$,
\begin{eqnarray}
 f(\x) =\c^T\x,\label{eq:linobj}
\end{eqnarray}
then the problem in (\ref{eq:lincons}) is called \emph{linear} program. Given that (\ref{eq:lincons}) is one of the most famous  classical optimization problems, it has a rich history and has been thoroughly studied for over a century (see, e.g., \cite{BV,Rock70,Luen84,Kuhn76,KuhnTuck51,BenNem01,BerTsi97,Bertsek99} and references therein). Depending on the context, types of applications, and overall usefulness many aspects of (\ref{eq:lincons}) are usually of interest for studying. For example, from an algorithmic point of view, most of the standard  optimization treatments (e.g., \cite{BV,Rock70,Luen84,Kuhn76,KuhnTuck51,BenNem01,BerTsi97,Bertsek99}), position just finding the minimal value of $f(\x)$ and $\x$'s that achieve it as the primary goal. Doing so computationally efficiently is usually assumed as a critically important complementary part of the primary goal package as well.

Our interest here is slightly different and is not primarily motivated by the classical algorithmic point of view. Instead of trying to algorithmically find the minimal $f(\x)$, we will be interested in scenarios where $f(\x)$ can be analytically characterized without actually explicitly solving (\ref{eq:lincons}). For that to be possible we switch discussion to random mediums and consider \emph{random linear programs} (rlps). We also restrict the focus to fully polyhedral/polytopal rlps, i.e. to rlps with inequality constraints. This basically means that the following version of (\ref{eq:lincons}) will be the main subject below
\begin{eqnarray}
\xi \triangleq \frac{1}{\sqrt{n}}\min_{\x\in\mR^{n}} & & \c^T\x \nonumber \\
\mbox{subject to} & & A\x\leq \a, \label{eq:linconsa0}
\end{eqnarray}
where $A$ is random. It is not that difficult to see that for $\c$ generated uniformly randomly on the unit $n$-dimensional sphere, $\mE \xi$ is the mean width of the polytope/polyhedron $A\x\leq \a$. In scenarios where $A$ and/or $\a$ are random, we view $\mE \xi$ as the mean width of the random polytope. Random optimization problems, and in particular the linear ones, have played an important role in various scientific and engineering fields and, consequently, have been the subject of an intensive studying over the last several decades. We below briefly discuss some of the most relevant known results.

\subsection{Relevant prior work}
\label{sec:priorwork}

Proving the average case polynomial complexity of famous linear program algorithms is among the most successful and well known examples of the utilization of rlps. In particular, \cite{Borgw82,Smale83} (see also \cite{Borgw99}) showed that the simplex method on average has a polynomial complexity. In \cite{SpiTeng04}, the so-called smoothed simplex analysis was introduced and it was shown that even close to deterministic linear program instances can typically be solved in polynomial time (for subsequent and more recent smoothed approach related developments see also, e.g., \cite{HuiLeeZhang23,DadHui18}).

One can think of smoothed analysis as being inspired by the early works on the so-called \emph{sensitivity} analysis and \emph{stochastic} programming. The idea was to determine the effect uncertainties or small perturbations in the problem's infrastructure (caused, say, by numerical errors or an overall lack of access to some portions of the system) would have on the objective value and solutions (see, e.g., \cite{Beale55,Dantzig55,FerDantzig56}). Early objective distributional effects considerations with (not necessarily always small) perturbations in any of $B$, $\b$, and $\c$ were discussed in \cite{Tintner60,Prekopa66,Babbar55}. More recent developments, with the perturbations particularly related to $\b$, can be found in, e.g., \cite{LiuBunNWl23,Klattetal22}, whereas more on general stochastic programming concepts can be found in classical reference, e.g., \cite{KallMayer79,RuzShapiro03}. Also, although not directly related to the linear programs, we mention that quite a lot of sensitivity or stochastic work has been done within the \emph{nonlinear} contexts as well (see, for example, excellent references \cite{DupWets88,Dupacova87,Shapiro89,Shapiro93}).

A, so to say, subclass of general linear programs are the feasibility linear problems. They also have a very rich history and we here focus on some of the most relevant works within the randomized contexts. Probably among the first random feasibility problems (rfps) that gained a strong prominence are the machine learning pattern recognition instances and in particular the so-called perceptrons. Perceptrons are often viewed as the most fundamental units of any neural network concept. Consequently, studying their various properties has been an attractive research topic for the better part of the last seven decades. The starting point were the famous  spherical perceptron capacity characterizations obtained in the early sixties of the last century \cite{Cover65,Wendel62} (the spherical perceptrons are obtained for $\c=\a=0$ in (\ref{eq:linconsa0})). They famously  established that the perceptron's  capacity (maximum allowable $m_1$  in (\ref{eq:linconsa0}) when $\c=\a=0$) basically \emph{doubles the dimension} of the data ambient space, $n$. After being initially proven as a remarkable combinatorial geometry fact in \cite{Cover65,Winder61,Wendel62}, decades later, it was reproved in various forms in a host of different fields ranging from machine learning and pattern recognition to probability and information theory (see, e.g., \cite{BalVen87,StojnicISIT2010binary,DTbern,Gar88,StojnicGardGen13}). Many more advanced results naturally followed. Some of the most relevant ones related to single perceptrons can be found in e.g., \cite{SchTir03,Talbook11a,Talbook11b,StojnicGardGen13,Stojnicnegsphflrdt23} and to networks of perceptrons in e.g.,
\cite{EKTVZ92,BHS92,MitchDurb89,BalMalZech19,ZavPeh21,Stojnictcmspnncapdinfdiffactrdt23}. In addition to the works related to the continuous rfps, a strong effort has been put forth in studying the corresponding discrete ones as well. Examples include binary perceptrons \cite{Gar88,KraMez89,DingSun19,NakSun23,BoltNakSunXu22,Tal99a,Stojnicbinperflrdt23}, various other forms of discrete perceptrons \cite{StojnicDiscPercp13,GutSte90}, general integer rfpos \cite{ChanVem14} and so on.

There has been a lot of great work related to the algorithmic aspects of feasibility problems as well. We here particularly emphasized the so-called Kczmarz iterative algorithm for linear programs. After its main contours were proposed in  \cite{Kacz37}, it was further developed in \cite{Agmon54,MotzShoen54}. Its excellent properties turned out to be particularly beneficial within the compressed sensing context and it was again brought to attention, and in a way revitalized, over the last a couple of decades (see, e.g., \cite{NeedellTropp14,StrVer09,EldNeed11} as well as \cite{Morshedetal22,LiuGu21} for the most recent developments).

Finally, characterizing the objective of (\ref{eq:linconsa0}) in a fully randomized context (not the above discussed one with the small perturbation type of randomness) has been studied as well. High-dimensional geometry type of studying of the mean width of the polytope was considered  in, e.g., \cite{AGPro15,GiaHioTs16,Gluskin88}.  For a very similar, symmetric polytope scenario, the lower and upper bounds differing by a constant factor were established in \cite{LPRTJ05}. The exact values of the generic proportionality constants from \cite{LPRTJ05} were then precisely determined in the limiting scenario $\alpha\rightarrow\infty$ in \cite{BakOstTik23}.

\subsection{Our contribution}
\label{sec:contribution}

We set $m\triangleq m_1$ and consider large $n$ \emph{linear} regime, i.e. the regime where $\alpha=\lim_{n\rightarrow\infty} \frac{m_1}{n}=\lim_{n\rightarrow\infty} \frac{m}{n}$  remains constant as $n$ grows. We then consider (\ref{eq:linconsa0}) in a typical statistical scenario with elements of $A$ being iid standard normals. Such scenarios are precisely the ones considered in \cite{StojnicRegRndDlt10,StojnicCSetam09,StojnicICASSP10var} where the foundational principles of the \emph{random duality theory} (RDT) were introduced. We here rely on these principles, utilize the RDT, and create a generic framework for the analysis of (\ref{eq:linconsa0}). As a consequence, we obtain the precise closed form characterization of $\xi$ (the objective value of (\ref{eq:linconsa0})) as a function of any real number $\alpha\in(0,\infty)$. Moreover, we show that for $\alpha\rightarrow\infty$ our results precisely match those obtained in \cite{BakOstTik23}.

In sections that follow below, we will slowly introduce all the needed RDT ingredients. We, however, find it useful to note here before starting the technical presentation, that our exposition will in no way use or rely on any analytical/algorithmic concepts particularly related to (\ref{eq:lincons}) or (\ref{eq:linconsa0}). In other words, besides a few well known generic optimization concepts, no extensive or advanced knowledge of the optimization theory is required.

\section{The problem setup} 
\label{sec:randlincons}

As emphasized earlier, we focus on programs (\ref{eq:lincons}) that have a linear objective and linear inequalities. In other words, we consider
\begin{eqnarray}
\xi=\frac{1}{\sqrt{n}}\min_{\x} & & \c^T\x \nonumber \\
\mbox{subject to} & & A\x\leq \a. \label{eq:randlincons1}
\end{eqnarray}
To avoid overwhelming the presentation with a tone of tiny details and special cases and to ultimately make the writing neater and easier, we assume throughout the paper that the objective in (\ref{eq:randlincons1}) exists and is either deterministically bounded (i.e., bounded for any realization of $A$) or that it is bounded with a probability (over the randomness of $A$) going $1$ as $n\rightarrow \infty$. Along the same lines, to further facilitate and make less cumbersome the overall exposition, we also assume boundedness (deterministic or random) of all other key quantities appearing in the derivations below. Equipped with such assumptions, we proceed by transforming (\ref{eq:randlincons1})
\begin{eqnarray}
\xi & = & \frac{1}{\sqrt{n}}\min_{\x}\max_{\lambda\geq 0} \lp \c^T\x+\lambda^TA\x -\lambda^T\a \rp. \label{eq:randlincons2}
\end{eqnarray}
 The above holds generically, i.e., for any $A$, $\a$, and $\c$. It therefore automatically applies to the random instances as well.
Given that such instances are of our particular interest here, the following summarizes the probabilistic characterization of $\xi$ that we are looking for.
 \begin{eqnarray}
\mbox{\bl{\textbf{Ultimate goal:}}}
\qquad\qquad \mbox{Given:} \quad && \alpha  =    \lim_{n\rightarrow \infty} \frac{m}{n}\in(0,\infty)  \nonumber \\
\quad \mbox{find} \quad & & \xi_{opt}(\alpha;\a) \nonumber \\
\mbox{such that} \quad  && \forall \epsilon>0, \quad \lim_{n\rightarrow\infty}\mP_A\lp (1-\epsilon)\xi_{opt}(\alpha;\a) \leq \xi \leq (1+\epsilon)\xi_{opt}(\alpha;\a) \rp\longrightarrow 1.\nonumber \\
  \label{eq:ex4}
\end{eqnarray}
As expected, $\alpha$ and $\epsilon$ can be arbitrarily large or small but they do not change as $n\rightarrow\infty$. Throughout the paper we adopt a usual convention where the subscripts next to $\mP$ and $\mE$ denote the randomness with respect to which the statistical evaluation is taken. When clear from the contexts, the subscripts are left unspecified.

\section{Handling random linear programs via RDT}
\label{sec:randlinconsrdt}

We start by recalling on the main RDT principles and then continue by showing, step-by-step, how each of them applies to the problems of our interest here.

\vspace{-.0in}\begin{center}
 	\tcbset{beamer,lower separated=false, fonttitle=\bfseries, coltext=black ,
		interior style={top color=yellow!20!white, bottom color=yellow!60!white},title style={left color=black!80!purple!60!cyan, right color=yellow!80!white},
		width=(\linewidth-4pt)/4,before=,after=\hfill,fonttitle=\bfseries}
 \begin{tcolorbox}[beamer,title={\small Summary of the RDT's main principles} \cite{StojnicCSetam09,StojnicRegRndDlt10}, width=1\linewidth]
\vspace{-.15in}
{\small \begin{eqnarray*}
 \begin{array}{ll}
\hspace{-.19in} \mbox{1) \emph{Finding underlying optimization algebraic representation}}
 & \hspace{-.0in} \mbox{2) \emph{Determining the random dual}} \\
\hspace{-.19in} \mbox{3) \emph{Handling the random dual}} &
 \hspace{-.0in} \mbox{4) \emph{Double-checking strong random duality.}}
 \end{array}
  \end{eqnarray*}}
\vspace{-.2in}
 \end{tcolorbox}
\end{center}\vspace{-.0in}
To make the presentation neat, we formalize all the key results (including both simple to more complicated ones) as lemmas and theorems.

\vspace{.1in}

\noindent \underline{1) \textbf{\emph{Algebraic objective characterization:}}}  The following lemma summarizes the algebraic discussions from previous sections and ultimately provides a convenient optimization representation of the objective $\xi$. Given that the trivial scaling by $\|\c\|_2$ of the objective in (\ref{eq:linconsa0}) introduces no conceptual changes, we below, without loss of generality, assume $\|\c\|_2=\sqrt{n}$.

\begin{lemma}(Algebraic optimization representation) Let $A\in\mR^{m\times n}$ be any given (possibly random) matrix. For a given $\c\in\mR^n$, $\|\c\|_2=\sqrt{n}$ and a vector $\a\in\mR^m$ with the components being fixed real numbers that do not change as $n\rightarrow\infty$, let $\xi$ be as in (\ref{eq:linconsa0}) and set
\begin{eqnarray}\label{eq:ta11}
f_{rp}(A) & \triangleq & \frac{1}{\sqrt{n}}\min_{\x}\max_{\lambda\geq 0} \lp \c^T\x+\lambda^TA\x -\lambda^T\a \rp
 \hspace{.8in} (\bl{\textbf{random primal}})
\nonumber \\
\xi_{rp} & \triangleq & \lim_{n\rightarrow\infty } \mE_A f_{rp}(A).   \end{eqnarray}
Then
\begin{equation}\label{eq:ta11a0}
\xi=f_{rp}(A) \quad \mbox{and} \quad \lim_{n\rightarrow\infty} \mE_A\xi =\xi_{rp}.
\end{equation}
\label{lemma:lemma1}
\end{lemma}
\begin{proof}
 Follows immediately via the Lagrangian from (\ref{eq:randlincons2}).
\end{proof}

The above lemma holds for any given matrix $A$. The RDT proceeds by imposing a statistics  on $A$.


\vspace{.1in}
\noindent \underline{2) \textbf{\emph{Determining the random dual:}}} As is standard within the RDT, we utilize the concentration of measure, which here basically means that for any fixed $\epsilon >0$,  one can write (see, e.g. \cite{StojnicCSetam09,StojnicRegRndDlt10,StojnicICASSP10var})
\begin{equation}
\lim_{n\rightarrow\infty}\mP_A\left (\frac{|f_{rp}(A)-\mE_A(f_{rp}(A))|}{\mE_A(f_{rp}(A))}>\epsilon\right )\longrightarrow 0.\label{eq:ta15}
\end{equation}
Another key RDT ingredient is the following so-called random dual theorem.
\begin{theorem}(Objective characterization via random dual) Assume the setup of Lemma \ref{lemma:lemma1} and let the components of $A$ be iid standard normals. Let $\g\in\mR^m$ and $\h\in\mR^n$ be vectors comprised of iid standard normals and set
\vspace{-.0in}
\begin{eqnarray}
  f_{rd}(\g,\h) & \triangleq &
  \frac{1}{\sqrt{n}}\min_{\x}\max_{\lambda\geq 0} \lp \c^T\x+\|\x\|_2\lambda^T\g+\|\lambda\|_2\h^T\x -\lambda^T\a \rp
   \hspace{.8in} (\bl{\textbf{random dual}})
  \nonumber \\
 \xi_{rd} & \triangleq & \lim_{n\rightarrow\infty} \mE_{\g,\h} f_{rd}(\g,\h)  .\label{eq:ta16}
\vspace{-.0in}\end{eqnarray}
One then has \vspace{-.0in}
\begin{eqnarray}
  \xi_{rd} & \triangleq & \lim_{n\rightarrow\infty} \mE_{\g,\h} f_{rd}(\g,\h)
  \leq
  \lim_{n\rightarrow\infty} \mE_{A} f_{rp}(A)  \triangleq  \xi_{rp}. \label{eq:ta16a0}
\vspace{-.0in}\end{eqnarray}
and
\begin{eqnarray}
 \lim_{n\rightarrow\infty}\mP_{\g,\h} \lp f_{rd}(\g,\h)\geq (1-\epsilon)\xi_{rd}\rp
 \leq  \lim_{n\rightarrow\infty}\mP_{A} \lp f_{rp}(A)\geq (1-\epsilon)\xi_{rd}\rp.\label{eq:ta17}
\end{eqnarray}
\label{thm:thm1}
\end{theorem}\vspace{-.17in}
\begin{proof}
  Follows as a direct application of the Gordon's probabilistic comparison theorem (see, e.g., Theorem B in \cite{Gordon88} as well as Theorem 1, Corollary 1, and Section 2.7.2 in \cite{Stojnicgscomp16}).
\end{proof}
%
%
%
%
%
\vspace{.1in}
\noindent \underline{3) \textbf{\emph{Handling the random dual:}}} After solving the inner maximization over $\lambda$  we have
\begin{eqnarray}
  f_{rd}(\g,\h)
  & \triangleq &
  \frac{1}{\sqrt{n}}\min_{\x}\max_{\lambda\geq 0} \lp \c^T\x+\|\x\|_2\lambda^T\g+\|\lambda\|_2\h^T\x -\lambda^T\a \rp\nonumber \\
  & = &
  \frac{1}{\sqrt{n}}\min_{\x}\max_{\|\lambda\|_2} \lp \c^T\x+\|\lambda\|_2\h^T\x  + \|\lambda\|_2 \| \max \lp \|\x\|_2\g-\a,0\rp \|_2 \rp.\label{eq:ta18a0}
\end{eqnarray}
Keeping the $\|\lambda\|_2$ fixed and minimizing over $\x$, we further find
\begin{eqnarray}
  f_{rd}(\g,\h)
   & = &
  \frac{1}{\sqrt{n}}\min_{\x}\max_{\|\lambda\|_2} \lp \c^T\x+\|\lambda\|_2\h^T\x  + \|\lambda\|_2 \| \max \lp \|\x\|_2\g-\a,0\rp \|_2 \rp
  \nonumber \\
  & = &
  \frac{1}{\sqrt{n}}\min_{\|\x\|_2}\max_{\|\lambda\|_2} \lp  -\|\x\|_2\|\c+\|\lambda\|_2\h\|_2  + \|\lambda\|_2 \| \max \lp \|\x\|_2\g-\a,0\rp \|_2 \rp.\label{eq:ta18a1}
\end{eqnarray}
Setting $x=\|\x\|_2$ and $\lambda_f=\|\lambda\|_2$ and keeping $x$ and $\lambda_f$ fixed,  we then first find
\begin{eqnarray}
  \lim_{n\rightarrow\infty} \mE_{\h} \frac{1}{\sqrt{n}} \|\c+\|\lambda\|_2\h\|_2
&  = &\lim_{n\rightarrow\infty} \sqrt{ \frac{\mE_{\h}\|\c+\|\lambda\|_2\h\|_2^2
}{n} } \nonumber \\
&  = &\lim_{n\rightarrow\infty} \sqrt{ \frac{\mE_{\h}\lp \|\c\|_2^2+2\|\lambda\|_2\c^T\h+\|\lambda\|_2^2\|\h\|_2^2\rp
}{n} } \nonumber \\
&  = &\lim_{n\rightarrow\infty} \sqrt{ \frac{\|\c\|_2^2+2\|\lambda\|_2\mE_{\h}\lp\c^T\h\rp+\|\lambda\|_2^2\mE_{\h}\|\h\|_2^2
}{n} } \nonumber \\
& =& \sqrt{1+\lambda_f^2}.\label{eq:ta18a2}
\end{eqnarray}
Moreover, we also have
\begin{eqnarray}
f^{(1)}(x;\a) & \triangleq  &
\lim_{n\rightarrow\infty} \mE_{\g} \frac{1}{\sqrt{n}} \| \max \lp \|\x\|_2\g-\a,0\rp \|_2 \nonumber \\
& = &
\lim_{n\rightarrow\infty} \mE_{\g} \frac{1}{\sqrt{n}} \| \max \lp x\g-\a,0\rp \|_2\nonumber \\
& = &
x\lim_{n\rightarrow\infty} \mE_{\g} \frac{1}{\sqrt{n}} \left \| \max \lp \g-\frac{\a}{x},0\rp \right \|_2\nonumber \\
& = &
x\lim_{n\rightarrow\infty} \sqrt{\frac{\mE_{\g} \left \| \max \lp \g-\frac{\a}{x},0\rp \right \|_2^2}{n}} \nonumber \\
& = &
x\lim_{n\rightarrow\infty} \sqrt{\frac{\sum_{i=1}^{n}\mE_{\g_i} \max \lp \g_i-\frac{\a_i}{x},0\rp^2}{n}} \nonumber \\
& = &
x\lim_{n\rightarrow\infty} \sqrt{\frac{\sum_{i=1}^{n}  \frac{1}{\sqrt{2\pi}}\int_{\frac{\a_i}{x}}^{\infty} \lp \g_i-\frac{\a_i}{x},0\rp^2e^{-\frac{\g_i^2}{2}}d\g_i}{n}} \nonumber \\& = &
x\lim_{n\rightarrow\infty} \sqrt{\frac{\sum_{i=1}^{n} f^{(1)}_i(x;\a_i)}{n}},\label{eq:ta18a3}
\end{eqnarray}
where
\begin{eqnarray}
f^{(1)}_i(x;\a_i)
& \triangleq &
 \frac{1}{\sqrt{2\pi}}\int_{\frac{\a_i}{x}}^{\infty} \lp \g_i-\frac{\a_i}{x},0\rp^2e^{-\frac{\g_i^2}{2}}d\g_i =
 \frac{1}{2} \lp\lp\frac{\a_i}{x}\rp^2 + 1\rp \erfc\lp\frac{\a_i}{x\sqrt{2}}\rp - \frac{\a_i}{x} \frac{e^{-\frac{\a_i^2}{2x^2}}}{\sqrt{2\pi}}.\label{eq:ta18a4}
\end{eqnarray}
A combination of (\ref{eq:ta18a1})-(\ref{eq:ta18a3}) together with measure concentrations gives
\begin{eqnarray}
 \mE_{\g,\h} f_{rd}(\g,\h)
   & = &
 \mE_{\g,\h}  \frac{1}{\sqrt{n}}\min_{\|\x\|_2}\max_{\|\lambda\|_2} \lp  -\|\x\|_2\|\c+\|\lambda\|_2\h\|_2  + \|\lambda\|_2 \| \max \lp \|\x\|_2\g-\a,0\rp \|_2 \rp \nonumber \\
   & = &
 \min_{x>0}\max_{\lambda_f>0} \lp  -x\sqrt{1+\lambda_f^2}  + \lambda_f  f^{(1)}(x;\a)\rp,
 \label{eq:ta18a5}
\end{eqnarray}
where
\begin{eqnarray}
 f^{(1)}(x;\a)  \triangleq  x\lim_{n\rightarrow\infty}\sqrt{\frac{\sum_{i=1}^{m}f^{(1)}_i(x,\a_i)}{n}}
 =
x \lim_{n\rightarrow\infty}\sqrt{\frac{\sum_{i=1}^{m}  \lp  \frac{1}{2} \lp\lp\frac{\a_i}{x}\rp^2 + 1\rp \erfc\lp\frac{\a_i}{x\sqrt{2}}\rp - \frac{\a_i}{x} \frac{e^{-\frac{\a_i^2}{2x^2}}}{\sqrt{2\pi}}\rp
   }{n}}.
 \label{eq:ta18a6}
\end{eqnarray}
To optimize over $\lambda_f$, we first determine the following derivative
\begin{eqnarray}
 \frac{d\lp  -x\sqrt{1+\lambda_f^2}  + \lambda_f  f^{(1)}(x;\a)\rp}{d\lambda_f}
 & = & -\frac{2x\lambda_f}{2\sqrt{1+\lambda_f^2}} +  f^{(1)}(x;\a).
 \label{eq:ta18a7}
\end{eqnarray}
Setting the above derivative to zero then gives
\begin{eqnarray}
  x^2\lambda_f^2= \lp 1+\lambda_f^2\rp \lp f^{(1)}(x;\a)\rp^2,
 \label{eq:ta18a8}
\end{eqnarray}
which implies that one has for the optimal $\lambda_f$, $\hat{\lambda}_f$,
\begin{eqnarray}
  \hat{\lambda}_f = \frac{f^{(1)}(x;\a)}{\sqrt{x^2-\lp f^{(1)}(x;\a)\rp^2}},
 \label{eq:ta18a9}
\end{eqnarray}
Plugging the above $\hat{\lambda}_f$ back in (\ref{eq:ta18a5}), we obtain
\begin{eqnarray}
 \mE_{\g,\h} f_{rd}(\g,\h)
    & = &
 \min_{x>0}\max_{\lambda_f>0} \lp  -x\sqrt{1+\lambda_f^2}  + \lambda_f  f^{(1)}(x;\a)\rp \nonumber \\
    & = &
 \min_{x>0}  \lp  -x\sqrt{1+\lp  \frac{f^{(1)}(x;\a)}{\sqrt{x^2-\lp f^{(1)}(x;\a)\rp^2}}     \rp^2}  + \frac{\lp f^{(1)}(x;\a)\rp^2}{\sqrt{x^2-\lp f^{(1)}(x;\a)\rp^2}}\rp \nonumber \\
     & = &
 \min_{x>0} \lp  -x\sqrt{ \frac{x^2}{x^2-\lp f^{(1)}(x;\a)\rp^2}}  + \frac{\lp f^{(1)}(x;\a)\rp^2}{\sqrt{x^2-\lp f^{(1)}(x;\a)\rp^2}}\rp \nonumber \\
      & = &
 \min_{x>0} \sqrt{x^2-\lp f^{(1)}(x;\a)\rp^2}
 \label{eq:ta18a10}
\end{eqnarray}

We summarize the above discussion in the following lemma.
\begin{lemma}(Characterization of random dual) Assume the setup of Theorem \ref{thm:thm1} and let $\xi_{rd}$ be as in (\ref{eq:ta16}). One then has \vspace{-.0in}
\begin{eqnarray}
  \xi_{rd} & \triangleq & \lim_{n\rightarrow\infty} \mE_{\g,\h} f_{rd}(\g,\h)
=
 \min_{x>0} \sqrt{x^2-\lp f^{(1)}(x;\a)\rp^2} \nonumber \\
 & = &
 \min_{x>0} \sqrt{x^2-    x^2\lim_{n\rightarrow\infty}  \frac{\sum_{i=1}^{m}   \lp \frac{1}{2} \lp\lp\frac{\a_i}{x}\rp^2 + 1\rp \erfc\lp\frac{\a_i}{x\sqrt{2}}\rp - \frac{\a_i}{x} \frac{e^{-\frac{\a_i^2}{2x^2}}}{\sqrt{2\pi}} \rp
   }{n} } \label{eq:lemma2ta16a0}
\end{eqnarray}
and
\begin{eqnarray}
 \lim_{n\rightarrow\infty}\mP_{\g,\h} \lp  (1-\epsilon)\xi_{rd} \leq  f_{rd}(\g,\h)\leq (1+\epsilon)\xi_{rd}\rp
\longrightarrow 1.\label{eq:lemma2ta17}
\end{eqnarray}
\label{lemma:lemma2}
\end{lemma}\vspace{-.17in}
\begin{proof}
Follows from the above discussion and trivial concentration of $f_{rd}(\g,\h)$.
\end{proof}

The above discussion allows us then to also formulate the following theorem.

\begin{theorem}(Optimal objective) Let $A\in\mR^{m\times n}$ be a random matrix with iid standard normal elements. Consider a large $n$ linear regime with $\alpha=\lim_{n\rightarrow\infty}\frac{m}{n}$ that does not change as $n$ grows. For a given $\c\in\mR^n$, $\|\c\|_2=\sqrt{n}$ and a vector $\a\in\mR^m$ with the components being fixed real numbers that do not change as $n\rightarrow\infty$, let $\xi$, $\xi_{rp}$, and $\xi_{rd}$ as in (\ref{eq:linconsa0}),
(\ref{eq:ta11}), and (\ref{eq:ta16}),  respectively. Then
\begin{equation}\label{eq:thm2a11a0}
 \lim_{n\rightarrow\infty} \mE_A\xi =\xi_{rp}=\xi_{rd}
 =
  \min_{x>0} \sqrt{x^2-   x^2 \lim_{n\rightarrow\infty}  \frac{\sum_{i=1}^{m}   \lp \frac{1}{2} \lp\lp\frac{\a_i}{x}\rp^2 + 1\rp \erfc\lp\frac{\a_i}{x\sqrt{2}}\rp - \frac{\a_i}{x} \frac{e^{-\frac{\a_i^2}{2x^2}}}{\sqrt{2\pi}} \rp
   }{n} } \triangleq \xi_{opt}(\alpha;\a).
\end{equation}
Moreover,
\begin{eqnarray}
 \lim_{n\rightarrow\infty}\mP_{A} \lp  (1-\epsilon)\mE_A\xi \leq  \xi \leq (1+\epsilon)\mE_A\xi\rp
\longrightarrow 1.\label{eq:thm2ta17}
\end{eqnarray}
In particular, for $\a=\1$ (where $\1$ is a column vector of all ones), we have
\begin{equation}\label{eq:thm2a11a1}
 \lim_{n\rightarrow\infty} \mE_A\xi =\xi_{rp}=\xi_{rd}
 =
  \min_{x>0} \sqrt{x^2-   x^2 \alpha \lp \frac{1}{2} \lp \frac{1}{x^2} + 1\rp \erfc\lp\frac{1}{x\sqrt{2}}\rp - \frac{1}{x} \frac{e^{-\frac{1}{2x^2}}}{\sqrt{2\pi}} \rp} \triangleq \xi_{opt}(\alpha).
\end{equation}
Moreover, for random $\a$ with identically distributed components $\a_i$, we have
\begin{equation}\label{eq:thm2a11a1}
 \lim_{n\rightarrow\infty} \mE_A\xi =\xi_{rp}=\xi_{rd}
 =
  \min_{x>0} \sqrt{x^2-   x^2 \alpha  \mE_{\a_i}\lp \frac{1}{2} \lp\lp\frac{\a_i}{x}\rp^2 + 1\rp \erfc\lp\frac{\a_i}{x\sqrt{2}}\rp - \frac{\a_i}{x} \frac{e^{-\frac{\a_i^2}{2x^2}}}{\sqrt{2\pi}} \rp}.
\end{equation}
\label{thm:thm2}
\end{theorem}
\begin{proof}
 We first have that 
 \begin{equation}\label{eq:proofthm2eq1}
 \lim_{n\rightarrow\infty} \mE_A\xi =\xi_{rp}\geq \xi_{rd}
 =
  \min_{x>0} \sqrt{x^2-   x^2 \lim_{n\rightarrow\infty}  \frac{\sum_{i=1}^{m}   \lp \frac{1}{2} \lp\lp\frac{\a_i}{x}\rp^2 + 1\rp \erfc\lp\frac{\a_i}{x\sqrt{2}}\rp - \frac{\a_i}{x} \frac{e^{-\frac{\a_i^2}{2x^2}}}{\sqrt{2\pi}} \rp
   }{n} },
\end{equation}
holds due to Theorem \ref{thm:thm1} (in particular, due to the inequality in (\ref{eq:ta16a0})) and (\ref{eq:lemma2ta16a0}). The reversal inequality in (\ref{eq:proofthm2eq1}) holds due to the underlying convexity and the fact that strong random duality is in place and all the reversal arguments from \cite{StojnicRegRndDlt10,StojnicGorEx10} apply.

We further have that (\ref{eq:thm2ta17}) is rewritten (\ref{eq:lemma2ta16a0}) and holds due to the concentration of $\xi$. One can observe that
\begin{eqnarray}
 \lim_{n\rightarrow\infty}\mP_{\g,\h} \lp  (1-\epsilon)\xi_{rd} \leq  f_{rd}(\g,\h)\leq (1+\epsilon)\xi_{rd}\rp
&  \leq & 
 \lim_{n\rightarrow\infty}\mP_{\g,\h} \lp f_{rd}(\g,\h)\geq (1-\epsilon)\xi_{rd}\rp \nonumber \\
& \leq & \lim_{n\rightarrow\infty}\mP_{A} \lp f_{rp}(A)\geq (1-\epsilon)\xi_{rd}\rp  \nonumber \\
& \leq & \lim_{n\rightarrow\infty}\mP_{A} \lp \xi \geq (1-\epsilon) \mE_{A}\xi\rp,\label{eq:proofthm2eq2}
\end{eqnarray}
where the first inequality follows by the probability axioms, the second by (\ref{eq:ta17}), and the third by a combination of (\ref{eq:ta11a0})
and (\ref{eq:thm2a11a0}). A combination of (\ref{eq:lemma2ta17}) and (\ref{eq:proofthm2eq2}) then gives
\begin{eqnarray}
\lim_{n\rightarrow\infty}\mP_{A} \lp \xi \geq (1-\epsilon) \mE_{A}\xi\rp \longrightarrow 1.\label{eq:proofthm2eq3}
\end{eqnarray}
The reversal arguments of \cite{StojnicRegRndDlt10,StojnicGorEx10} then ensure that 
\begin{eqnarray}
\lim_{n\rightarrow\infty}\mP_{A} \lp \xi \leq (1+\epsilon) \mE_{A}\xi\rp \longrightarrow 1.\label{eq:proofthm2eq4}
\end{eqnarray}
A combination of (\ref{eq:proofthm2eq3}) and (\ref{eq:proofthm2eq4}) confirms the concentration (\ref{eq:thm2ta17}). Keeping in mind that particularization $\a_i=1$ trivially specializes (\ref{eq:thm2a11a0}) to (\ref{eq:thm2a11a1}) completes the proof.
\end{proof}

\vspace{.1in}
\noindent \underline{ 4) \textbf{\emph{Double checking strong random duality:}}} As we have mentioned above, due to the underlying convexity, the reversal arguments of \cite{StojnicGorEx10,StojnicRegRndDlt10} and the strong random duality are in place as well, and the results proven in the above theorem are not only excellent bounds but also the exact characterizations.

\subsection{Numerical results}
\label{secnumres}

The results obtained based on the above theorem are shown in Table \ref{tab:tab1}. We in parallel show the results obtained based on numerical simulations in \cite{BakOstTik23}. Even though our proofs require $n\rightarrow\infty$, we, already for a fairly small $n=50$, observe a rather remarkable agreement (with relative difference $\sim 0.1\%$) between the theory and simulations.
\begin{table}[h]
\caption{Gaussian polyhedron $\a=\1$; Optimal objective $\xi_{opt}(\alpha)=\mE_A\xi$ --- \textbf{\bl{theory}} versus \textbf{simulations}}
\vspace{.1in}
\centering
\def\arraystretch{1.2}
\begin{tabular}{||c||c|c||c||c||}\hline\hline
 \hspace{-0in} dimensions ratio                                               & \# of constraints & \# of unknowns & \bl{$\mE_A\xi$}        & $\mE_A\xi$    \\
 \hspace{-0in}$\alpha=\frac{m}{n}$                                               & $m$ & $n$ & $\textbf{\bl{theory (this paper)}}$        & $\textbf{simulations \cite{BakOstTik23}}$    \\ \hline\hline
$20$                                        & $  1000 $  & $   50  $  & $ \bl{\mathbf{0.50402}} $  & $ \mathbf{0.50626} $  \\ \hline \hline
$40$                                        & $  2000 $  & $   50  $  &  $ \bl{\mathbf{0.43907}} $ & $  \mathbf{0.44119} $ \\  \hline
 $120$                                      & $ 6000  $ & $  50 $ &  $ \bl{\mathbf{0.37264}} $     & $  \mathbf{0.37256} $   \\ \hline\hline
 $200$                                      & $ 10000  $ & $  50 $ & $  \bl{\mathbf{0.35032}} $       & $ \mathbf{0.35176} $   \\ \hline\hline
 $400$                                      & $ 20000  $ & $  50 $ & $  \bl{\mathbf{0.32545}} $    & $  \mathbf{0.32473} $   \\ \hline\hline
  \end{tabular}
\label{tab:tab1}
\end{table}

The obtained results are also visualized in Figures \ref{fig:fig1}. Also from Theorem \ref{thm:thm2}, we have
\begin{equation}\label{eq:alpjainf}
\lim_{\alpha\rightarrow\infty}\frac{\xi_{opt}(\alpha)}{\sqrt{2\log\lp\alpha\rp}}=1,  
\end{equation}
which matches what was obtained in \cite{BakOstTik23}. In Figure \ref{fig:fig1}, the limiting curve, $\frac{1}{\sqrt{2\log\lp\alpha\rp}}$, is displayed as well.

\begin{figure}[h]
\centering
\centerline{\includegraphics[width=1\linewidth]{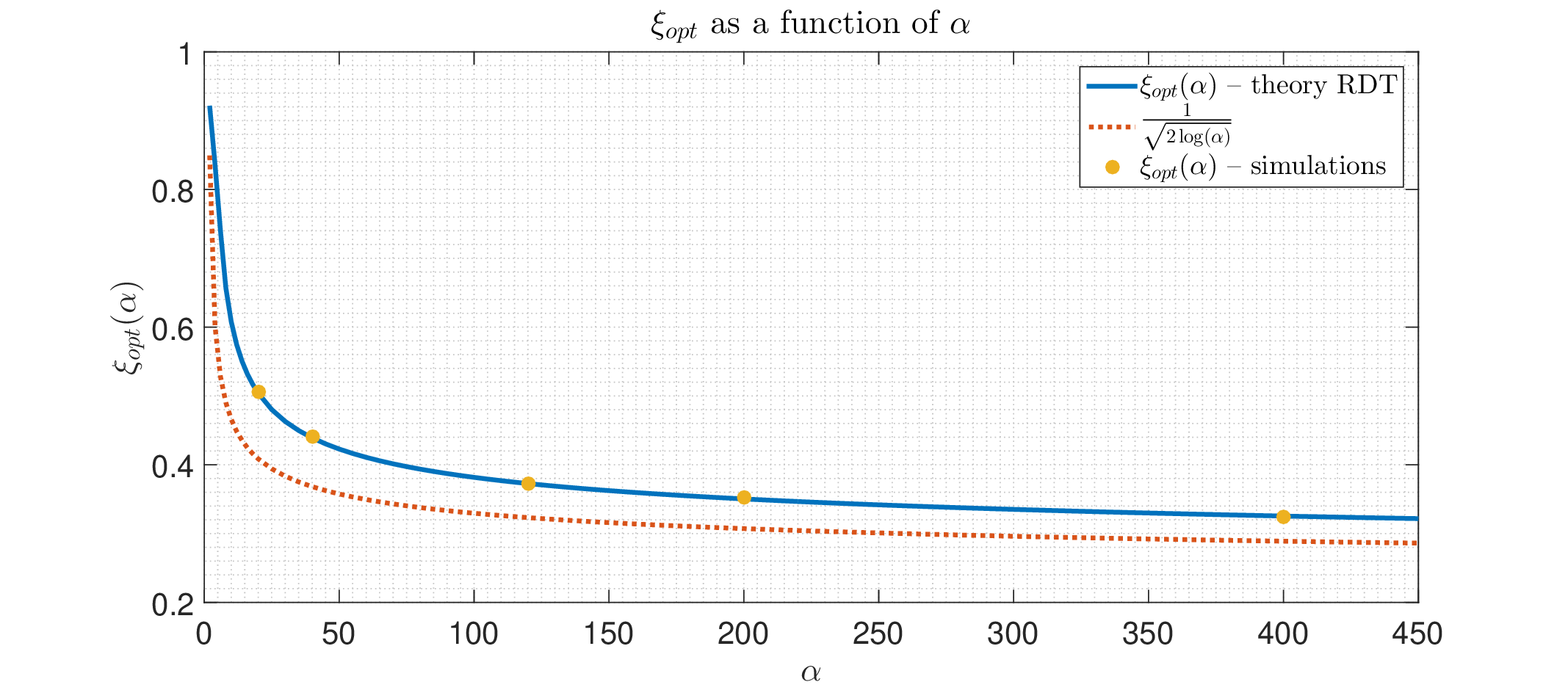}}
\caption{Gaussian polyhedron -- Optimal objective $\xi_{opt}(\alpha)=\mE_A\xi$}
\label{fig:fig1}
\end{figure}

\section{Conclusion}
\label{sec:conc}

We considered \emph{random linear programs} (rlps) which are a subclass  of \emph{random optimization problems} (rops). It was demonstrated how the powerful machinery of \emph{random duality theory} (RDT) \cite{StojnicRegRndDlt10} can be utilized to create a generic framework for the analysis of rlps. We then applied the framework for the analysis of the optimization of a linear objective over the standard Gaussian polyhedrons. For a typically hardest to analyze, so-called large $n$ linear (proportional) regime, we obtained the closed form exact characterization of the program's objective value as a function of a single system parameter $\alpha$ which represents the ratio of the number of polyhedral constraints and the problem's ambient dimension. We also uncovered that the result is insensitive with respect to the particular linearity of the objective, which, due to the rotational symmetry of the matrix of constraints, is to be expected. Moreover, if the vector defining the hyperplane of the objective is chosen uniformly at random from the unit sphere, the corresponding averaged objective value is precisely the mean width of the random polyhedron/polytope. This effectively means that our results automatically provide the characterization of the mean width of the Gaussian's polytopes. Our characterizations are valid for any number of cconstraint/dimension ratios $\alpha$, including the limiting $\alpha\rightarrow\infty$ one. Moreover, in the
 limiting scenario, our results take a particularly simple $\frac{1}{\sqrt{2\log(\alpha)}}$ form and precisely match those obtained for $\alpha\rightarrow\infty$ in \cite{BakOstTik23}.
 
The presented methodology is very generic and many extensions and/or generalizations are possible. For example, besides optimal values of the objective various other properties of optimization problems are of interest. These may relate to behavior of the optimal solutions or critical constraints. Moreover, the randomness can appear in different forms (perturbation or full scale) and can affect different parts of the problem infrastructure. Moreover, generalization are in no way restricted to  \emph{random linear problems} (rlps). They apply equally well to general \emph{random optimization problems} (rops). Various other random structures can be analyzed as well, including, for example, \emph{random feasibility problems} (rfps) (although pretty detailed, the list of random structures discussed in, e.g., \cite{Stojnicsflgscompyx23} is only a small sample of a large collection of  problems that can be handled through the machinery presented here). The associated technical details are often problem specific,  and we discuss them in separate papers.

As  pointed out in \cite{StojnicRegRndDlt10}, the RDT considerations do not require the standard Gaussianity  assumption. Utilizing the assumption, however, made the presentation smoother. While the writing would necessarily be more cumbersome, the needed conceptual adjustments as one deviates from the Gaussianity are rather minimal. In particular, the Lindeberg central limit theorem variant (see, e.g.,  \cite{Lindeberg22}) can be utilized to extend the RDT results to various different statistics. We consider the approach of \cite{Chatterjee06} as particularly elegant in that regard.

\begin{singlespace}
\bibliographystyle{plain}
\bibliography{nflgscompyxRefs}

\begin{thebibliography}{10}

\bibitem{Agmon54}
S.~Agmon.
\newblock The relaxation method for linear inequalities.
\newblock {\em Canadian Journal of Mathematics}, 6:382--392, 1954.

\bibitem{AGPro15}
D.~Alonso-Gutierrez and J.~Prochno.
\newblock On the {G}aussian behavior of marginals and the mean width of random
  polytopes.
\newblock {\em Proc. Amer. Math. Soc}, 143(2):821--832, 2015.

\bibitem{Babbar55}
M.~M. Babbar.
\newblock Distributions of solutions of a set of linear equations (with an
  application to linear programming).
\newblock {\em J. Amer. Statist. Assoc.}, 50(1):854--869, 1955.

\bibitem{BakOstTik23}
M.~Bakhshi, J.~Ostrowski, and K.~Tikhomirov.
\newblock On the optimal objective value of random linear programs.
\newblock 2023.
\newblock available online at
  {\small\bl{\url{http://arxiv.org/abs/2401.17530}}}.

\bibitem{BalMalZech19}
C.~Baldassi, E.~M. Malatesta, and R.~Zecchina.
\newblock Properties of the geometry of solutions and capacity of multilayer
  neural networks with rectified linear unit activations.
\newblock {\em Phys. Rev. Lett.}, 123:170602, October 2019.

\bibitem{BalVen87}
P.~Baldi and S.~Venkatesh.
\newblock Number od stable points for spin-glasses and neural networks of
  higher orders.
\newblock {\em Phys. Rev. Letters}, 58(9):913--916, Mar. 1987.

\bibitem{BHS92}
E.~Barkai, D.~Hansel, and H.~Sompolinsky.
\newblock Broken symmetries in multilayered perceptrons.
\newblock {\em Phys. Rev. A}, 45(6):4146, March 1992.

\bibitem{Beale55}
E.~M. Beale.
\newblock On minimizing a convex function subject to linear inequalities.
\newblock {\em Journal of the Royal Statistical Society: Series B
  (Methodological)}, 17(2):173--184, 1955.

\bibitem{BenNem01}
A.~Ben-Tal and A.~Nemirovski.
\newblock {\em Lectures on Modern Convex Optimization. Analysis, Algorithms,
  and Engineering Applications}.
\newblock Society for Industrial and Applied Mathematics, 2001.

\bibitem{Bertsek99}
D.~P. Bertsekas.
\newblock {\em Nonlinear Programming}.
\newblock Athena Scientific, second edition, 1999.

\bibitem{BerTsi97}
D.~Bertsimas and J.~N. Tsitsiklis.
\newblock {\em Introduction to Linear Optimization}.
\newblock Athena Scientific, 1997.

\bibitem{BoltNakSunXu22}
E.~Bolthausen, S.~Nakajima, N.~Sun, and C.~Xu.
\newblock Gardner formula for {I}sing perceptron models at small densities.
\newblock {\em {P}roceedings of {T}hirty {F}ifth {C}onference on {L}earning
  {T}heory, {PMLR}}, 178:1787--1911, 2022.

\bibitem{Borgw82}
K.-H. Borgwardt.
\newblock The average number of pivot steps required by the simplex-method is
  polynomial.
\newblock {\em Z. Oper. Res. Ser. A-B}, 26(5):A157--A177, 1982.

\bibitem{Borgw99}
K.-H. Borgwardt.
\newblock A sharp upper bound for the expected number of shadow vertices in
  lp-polyhedra under orthogonal projection on two-dimensional planes.
\newblock {\em Math. Oper. Res.}, 24(3):544--603, 1999.

\bibitem{BV}
S.~Boyd and L.~Vandenberghe.
\newblock {\em Convex Optimization}.
\newblock Cambridge University Press, 2003.

\bibitem{ChanVem14}
K.~Chandrasekaran and S.~S. Vempala.
\newblock Integer feasibility of random polytopes: random integer programs.
\newblock {\em In Proceedings of the 5th conference on {I}nnovations in
  theoretical computer science}, pages 449--458, 2014.

\bibitem{Chatterjee06}
S.~Chatterjee.
\newblock A generalization of the {L}indenberg principle.
\newblock {\em The Annals of Probability}, 34(6):2061--2076.

\bibitem{Cover65}
T.~Cover.
\newblock Geomretrical and statistical properties of systems of linear
  inequalities with applications in pattern recognition.
\newblock {\em IEEE Transactions on Electronic Computers}, (EC-14):326--334,
  1965.

\bibitem{DadHui18}
D.~Dadush and S.~Huiberts.
\newblock A friendly smoothed analysis of the simplex method.
\newblock {\em STOC 18-Proceedings of the 50th Annual ACM Symposium on Theory
  of Computing}, pages 390--403, 2018.

\bibitem{Dantzig55}
G.~B. Dantzig.
\newblock Linear programming under uncertainty.
\newblock {\em Management Science}, 1:197--206, 1955.

\bibitem{DingSun19}
J.~Ding and N.~Sun.
\newblock Capacity lower bound for the {I}sing perceptron.
\newblock {\em {STOC} 2019: Proceedings of the 51st Annual {ACM SIGACT}
  {S}ymposium on {T}heory of {C}omputing}, pages 816--827, 2019.

\bibitem{DTbern}
D.~Donoho and J.~Tanner.
\newblock Counting the face of randomly projected hypercubes and orthants, with
  application.
\newblock {\em Discrete and Computational Geometry}, 43:522--541, 2010.

\bibitem{Dupacova87}
J.~Dupacova.
\newblock Stochastic programming with incomplete information: a surrey of
  results on postoptimization and sensitivity analysis.
\newblock {\em Optimization}, 18(4):507--532, 1987.

\bibitem{DupWets88}
J.~Dupacova and R.~Wets.
\newblock Asymptotic behavior of statistical estimators and of optimal
  solutions of stochastic optimization problems.
\newblock {\em The Annals of Statistics}, pages 1517--1549, 1988.

\bibitem{EldNeed11}
Y.~C. Eldar and D.~Needell.
\newblock Acceleration of randomized {K}aczmarz method via the
  {J}ohnson-{L}indenstrauss lemma.
\newblock {\em Numer. Algorithms}, 58(2):163--177, 2011.

\bibitem{EKTVZ92}
A.~Engel, H.~M. Kohler, F.~Tschepke, H.~Vollmayr, and A.~Zippelius.
\newblock Storage capacity and learning algorithms for two-layer neural
  networks.
\newblock {\em Phys. Rev. A}, 45(10):7590, May 1992.

\bibitem{FerDantzig56}
A.~R. Ferguson and G.~B. Dantzig.
\newblock The allocation of aircraft to routes: An example of linear
  programming under uncertain demand.
\newblock {\em Management Science}, 3(1):45--73, 1956.

\bibitem{MitchDurb89}
R.~M.~Durbin G.~J.~Mitchison.
\newblock Bounds on the learning capacity of some multi-layer networks.
\newblock {\em Biological Cybernetics}, 60:345--365, 1989.

\bibitem{Gar88}
E.~Gardner.
\newblock The space of interactions in neural networks models.
\newblock {\em J. Phys. A: Math. Gen.}, 21:257--270, 1988.

\bibitem{GiaHioTs16}
A.~Giannopoulos, L.~Hioni, and A.~Tsolomitis.
\newblock Asymptotic shape of the convex hull of isotropic log-concave random
  vectors.
\newblock {\em Adv. in Appl. Math.}, 75:116--143, 2016.

\bibitem{Gluskin88}
E.~D. Gluskin.
\newblock Extremal properties of orthogonal parallelepipeds and their
  applications to the geometry of banach spaces.
\newblock {\em Mat. Sb. (N.S.)}, 136(178)(1):85--96, 1988.

\bibitem{Gordon88}
Y.~Gordon.
\newblock On {M}ilman's inequality and random subspaces which escape through a
  mesh in ${R}^n$.
\newblock {\em Geometric Aspect of of functional analysis, Isr. Semin. 1986-87,
  Lect. Notes Math}, 1317, 1988.

\bibitem{GutSte90}
H.~Gutfreund and Y.~Stein.
\newblock Capacity of neural networks with discrete synaptic couplings.
\newblock {\em J. Physics A: Math. Gen}, 23:2613, 1990.

\bibitem{HuiLeeZhang23}
S.~Huiberts, Y.~T. Lee, and X.~Zhang.
\newblock Upper and lower bounds on the smoothed complexity of the simplex
  method.
\newblock {\em STOC 23-Proceedings of the 55th Annual ACM Symposium on Theory
  of Computing}, pages 1904--1917, 2023.

\bibitem{Kacz37}
S~Kaczmarz.
\newblock Angenaherte auflosung von systemen linearer gleichungen.
\newblock {\em Bulletin international de lacademie polonaise des sciences et
  des lettres}, 1937.

\bibitem{KallMayer79}
P.~Kall and J.~Mayer.
\newblock {\em Stochastic linear programming}, volume~7.
\newblock Springer, 1979.

\bibitem{Klattetal22}
M.~Klatt, A.~Munk, and Y.~Zemel.
\newblock Limit laws for empirical optimal solutions in random linear programs.
\newblock {\em Annals of Operations Research}, 315(1):251--278, 2022.

\bibitem{KraMez89}
W.~Krauth and M.~Mezard.
\newblock Storage capacity of memory networks with binary couplings.
\newblock {\em J. Phys. France}, 50:3057--3066, 1989.

\bibitem{Kuhn76}
H.~W. Kuhn.
\newblock Nonlinear programming. {A} historical view.
\newblock In R.~W. Cottle and C.~E. Lemke, editors, {\em Nonlinear
  Programming}, volume~9 of {\em {SIAM-AMS} {P}roceedings}. American
  Mathematical Society, 1976.

\bibitem{KuhnTuck51}
H.~W. Kuhn and A.~W. Tucker.
\newblock Nonlinear programming.
\newblock In J.~Neyman, editor, {\em Proceedings of the Second Berkeley
  Symposium on Mathematical Statistics and Probability}, pages 481--492.
  University of California Press, 1951.

\bibitem{Lindeberg22}
J.~W. Lindeberg.
\newblock Eine neue herleitung des exponentialgesetzes in der
  wahrscheinlichkeitsrechnung.
\newblock {\em Math. Z.}, 15:211--225, 1922.

\bibitem{LPRTJ05}
A.~E. Litvak, A.~Pajor, M.~Rudelson, and N.~Tomczak-Jaegermann.
\newblock Smallest singular value of random matrices and geometry of random
  polytopes.
\newblock {\em Adv. Math.}, 195(2):491--523, 2005.

\bibitem{LiuBunNWl23}
S.~Liu, F.~Bunea, and J.~Niles-Weed.
\newblock Asymptotic confidence sets for random linear programs.
\newblock 2023.
\newblock available online at
  {\small\bl{\url{http://arxiv.org/abs/2302.12364}}}.

\bibitem{LiuGu21}
Y.~Liu and C.~Q. Gu.
\newblock On greedy randomized block {K}aczmarz method for consistent linear
  systems.
\newblock {\em Linear Algebra Appl.}, 616:178--200, 2021.

\bibitem{Luen84}
D.~G. Luenberger.
\newblock {\em Linear and Nonlinear Programming}.
\newblock Addison-Wesley, second edition, 1984.

\bibitem{Morshedetal22}
M.~S. Morshed, M.~S. Islam, and M.~Noor-E-Alam.
\newblock Sampling {K}aczmarz-{M}otzkin method for linear feasibility problems:
  generalization and acceleration.
\newblock {\em Math. Program.}, 194(1-2):719--779, 2022.

\bibitem{MotzShoen54}
T.~S. Motzkin and I.~J. Schoenberg.
\newblock The relaxation method for linear inequalities.
\newblock {\em Canadian Journal of Mathematics}, 6:393--404, 1954.

\bibitem{NakSun23}
S.~Nakajima and N.~Sun.
\newblock Sharp threshold sequence and universality for {I}sing perceptron
  models.
\newblock {\em {P}roceedings of the 2023 {A}nnual {ACM-SIAM} {S}ymposium on
  {D}iscrete {A}lgorithms ({SODA})}, pages 638--674, 2023.

\bibitem{NeedellTropp14}
D.~Needell and J.~A. Tropp.
\newblock Paved with good intentions: analysis of a randomized block {K}aczmarz
  method.
\newblock {\em Linear Algebra Appl.}, 441:199--221, 2014.

\bibitem{Prekopa66}
A.~Prekopa.
\newblock On the probability distribution of the optimum of a random linear
  program.
\newblock {\em SIAM Journal on Control}, 4(1):211--222, 1966.

\bibitem{Rock70}
R.~T. Rockafellar.
\newblock {\em Convex Analysis}.
\newblock Princeton University Press, 1970.

\bibitem{RuzShapiro03}
A.~Ruszczynski and A.~Shapiro.
\newblock Stochastic programming models.
\newblock {\em Handbooks in Operations Research and Management Science}, pages
  841--858, 1989.

\bibitem{Shapiro89}
A.~Shapiro.
\newblock Asymptotic properties of statistical estimators in stochastic
  programming.
\newblock {\em The Annals of Statistics}, pages 841--858, 1989.

\bibitem{Shapiro93}
A.~Shapiro.
\newblock Asymptotic behavior of optimal solutions in stochastic programming.
\newblock {\em Mathematics of Operations Research}, 18(4):829--845, 1993.

\bibitem{SchTir03}
M.~Shcherbina and B.~Tirozzi.
\newblock Rigorous solution of the {G}ardner problem.
\newblock {\em Comm. on Math. Physics}, (234):383--422, 2003.

\bibitem{Smale83}
S.~Smale.
\newblock On the average number of steps of the simplex method of linear
  programming.
\newblock {\em Math. Programming}, 27(3):241--262, 1983.

\bibitem{SpiTeng04}
D.~A. Spielman and S.-H. Teng.
\newblock Smoothed analysis of algorithms: why the simplex algorithm usually
  takes polynomial time.
\newblock {\em J. ACM}, 51(3):385--468, 2004.

\bibitem{StojnicCSetam09}
M.~Stojnic.
\newblock Various thresholds for $\ell_1$-optimization in compressed sensing.
\newblock available online at \bl{\url{http://arxiv.org/abs/0907.3666}}.

\bibitem{StojnicICASSP10var}
M.~Stojnic.
\newblock $\ell_1$ optimization and its various thresholds in compressed
  sensing.
\newblock {\em ICASSP, IEEE International Conference on Acoustics, Signal and
  Speech Processing}, pages 3910--3913, 14-19 March 2010.
\newblock Dallas, TX.

\bibitem{StojnicISIT2010binary}
M.~Stojnic.
\newblock Recovery thresholds for $\ell_1$ optimization in binary compressed
  sensing.
\newblock {\em ISIT, IEEE International Symposium on Information Theory}, pages
  1593 -- 1597, 13-18 June 2010.
\newblock Austin, TX.

\bibitem{StojnicGardGen13}
M.~Stojnic.
\newblock Another look at the {G}ardner problem.
\newblock 2013.
\newblock available online at \bl{\url{http://arxiv.org/abs/1306.3979}}.

\bibitem{StojnicDiscPercp13}
M.~Stojnic.
\newblock Discrete perceptrons.
\newblock 2013.
\newblock available online at \bl{\url{http://arxiv.org/abs/1303.4375}}.

\bibitem{StojnicGorEx10}
M.~Stojnic.
\newblock Meshes that trap random subspaces.
\newblock 2013.
\newblock available online at \bl{\url{http://arxiv.org/abs/1304.0003}}.

\bibitem{StojnicRegRndDlt10}
M.~Stojnic.
\newblock Regularly random duality.
\newblock 2013.
\newblock available online at \bl{\url{http://arxiv.org/abs/1303.7295}}.

\bibitem{Stojnicgscomp16}
M.~Stojnic.
\newblock Generic and lifted probabilistic comparisons -- max replaces minmax.
\newblock 2016.
\newblock available online at \bl{\url{http://arxiv.org/abs/1612.08506}}.

\bibitem{Stojnicsflgscompyx23}
M.~Stojnic.
\newblock Bilinearly indexed random processes -- {\emph{stationarization}} of
  fully lifted interpolation.
\newblock 2023.
\newblock available online at \bl{\url{http://arxiv.org/abs/2311.18097}}.

\bibitem{Stojnicbinperflrdt23}
M.~Stojnic.
\newblock Binary perceptrons capacity via fully lifted random duality theory.
\newblock 2023.
\newblock available online at \bl{\url{http://arxiv.org/abs/2312.00073}}.

\bibitem{Stojnicnegsphflrdt23}
M.~Stojnic.
\newblock {Fl RDT} based ultimate lowering of the negative spherical perceptron
  capacity.
\newblock 2023.
\newblock available online at \bl{\url{http://arxiv.org/abs/2312.16531}}.

\bibitem{Stojnictcmspnncapdinfdiffactrdt23}
M.~Stojnic.
\newblock Exact capacity of the \emph{wide} hidden layer treelike neural
  networks with generic activations.
\newblock 2024.
\newblock available online at \bl{\url{http://arxiv.org/abs/2402.05719}}.

\bibitem{StrVer09}
T.~Strohmer and R.~Vershynin.
\newblock A randomized kaczmarz algorithm with exponential convergence.
\newblock {\em J. Fourier Anal. Appl.}, 15(2):262--278, 2004.

\bibitem{Tal99a}
M.~Talagrand.
\newblock Intersecting random half cubes.
\newblock {\em Random {S}tructures {A}lgorithms}, 15(3-4):436--449, 1999.

\bibitem{Talbook11b}
M.~Talagrand.
\newblock {\em Mean field models and spin glasse: {V}olume {II}}.
\newblock A series of modern surveys in mathematics 55, Springer-Verlag, Berlin
  Heidelberg, 2011.

\bibitem{Talbook11a}
M.~Talagrand.
\newblock {\em Mean field models and spin glasses: {V}olume {I}}.
\newblock A series of modern surveys in mathematics 54, Springer-Verlag, Berlin
  Heidelberg, 2011.

\bibitem{Tintner60}
G.~Tintner.
\newblock A note on stochastic linear programming).
\newblock {\em Econometrica}, 28(1):490--495, 1960.

\bibitem{Wendel62}
J.~G. Wendel.
\newblock A problem in geometric probability.
\newblock {\em Mathematica Scandinavica}, 1:109--111, 1962.

\bibitem{Winder61}
R.~O. Winder.
\newblock Single stage threshold logic.
\newblock {\em Switching circuit theory and logical design}, pages 321--332,
  Sep. 1961.
\newblock AIEE Special publications S-134.

\bibitem{ZavPeh21}
J.~A. Zavatone-Veth and C.~Pehlevan.
\newblock Activation function dependence of the storage capacity of treelike
  neural networks.
\newblock {\em Phys. Rev. E}, 103:L020301, February 2021.

\end{thebibliography}
\end{singlespace}

\end{document}